\documentclass[11pt]{amsart} 

\usepackage{amsmath, amsthm, amscd, amsfonts, amssymb, enumerate, verbatim, newlfont, calc, graphicx, color}
\usepackage[bookmarksnumbered, colorlinks, plainpages]{hyperref}
\usepackage{tikz}
 \usepackage{doi}

\usepackage{amsmath, amsthm, amscd, amsfonts, amssymb, enumerate,verbatim, newlfont, calc, graphicx, color}
\usepackage[bookmarksnumbered, colorlinks, plainpages]{hyperref}
\usepackage{tikz} 
\usepackage{doi}
\usepackage{cancel}
\newtheorem{theorem}{Theorem}[section] 
 
\newtheorem{fact}[theorem]{Fact}

\newtheorem{lemma}[theorem]{Lemma}
\newtheorem{proposition}[theorem]{Proposition}

\newtheorem{definition}[theorem]{Definition}

\newtheorem{remark}[theorem]{Remark}

\newtheorem{example}[theorem]{Example}

\usepackage{mathtools}

\usepackage{pstricks}
\usepackage{epsfig}
\usepackage{pst-grad}
\usepackage{pst-plot}
\usepackage{subfig}

 \usepackage{academicons}

\begin{document}

\author[M. Nasernejad   and   J. Toledo]{Mehrdad  Nasernejad$^{1,2,*}$ and   Jonathan Toledo$^{3}$}
\title[Strong persistence index and fluctuations in colon powers]{Strong persistence index and fluctuations in colon powers of monomial ideals}

\thanks{$^*$Corresponding author}

%\thanks{Mehrdad Nasernejad  }

\thanks{E-mail addresses:  m$\_$nasernejad@yahoo.com  and  jonathan.toledo@infotec.mx}  
\maketitle

\begin{center}
{\it
$^{1}$Univ. Artois, UR 2462, Laboratoire de Math\'{e}matique de  Lens (LML), \\  F-62300 Lens, France \\ 
$^{2}$Universit\'e  Caen Normandie, ENSICAEN, CNRS, Normandie Univ, GREYC UMR  6072, F-14000 Caen,  France\\
$^{3}$INFOTEC Centro de investigaci\'{o}n e innovaci\'{o}n en informaci\'{o}n \\
 y comunicaci\'{o}n, Ciudad de M\'{e}xico,14050, M\'{e}xico
}
\end{center}

\vspace{0.4cm}

\begin{abstract}
Let $I$ be an ideal in  a commutative  Noetherian ring $R$. We say that a positive integer  $\ell_0$ is the \textit{strong persistence index} of $I$ if $\ell_0$ is the smallest integer such that $(I^{\ell+1}:_RI)=I^{\ell}$ for all  $\ell\geq \ell_0$. The first aim of this paper is to study this notion for monomial ideals. 
We also  introduce the notion of  \textit{fluctuation in colon powers}
 if there exist  positive integers $a<b<c$ such that at least one of the following cases occurs:
\begin{itemize}
\item[(i)]  $(I^{a}:I) = I^{a-1}, \quad (I^{b}:I) \neq I^{b-1}, \quad \text{but} \quad (I^{c}:I) = I^{c-1}.$
\item[(ii)] $(I^{a}:I) \neq I^{a-1}, \quad (I^{b}:I) = I^{b-1}, \quad \text{but} \quad (I^{c}:I) \neq I^{c-1}.$
\end{itemize}
The second purpose  of this work is to study  this phenomenon for monomial ideals.  
 \end{abstract}
\vspace{0.4cm}

%%%%%%%%%%%%%%%%%%%%%%%%%%%%%%%%%%%%%%%%%%%%%%%%%%%%%%%%%%%%
%%%%%%%%%%%%%%%%%%%%%%%%%%%%%%%%%%%%%%%%%%%%%%%%%%%%%%%%%%%%
%%%%%%%%%%%%%%%%%%%%%%%%%%%%%%%%%%%%%%%%%%%%%%%%%%%%%%%%%%%%
%%%%%%%%%%%%%%%%%%%%%%%%%%%%%%%%%%%%%%%%%%%%%%%%%%%%%%%%%%%%
%%%%%%%%%%%%%%%%%%%%%%%%%%%%%%%%%%%%%%%%%%%%%%%%%%%%%%%%%%%%
%%%%%%%%%%%%%%%%%%%%%%%%%%%%%%%%%%%%%%%%%%%%%%%%%%%%%%%%%%%%

\section{Introduction and Overview}

Let $I$ be an ideal in a commutative  Noetherian ring $R$. Then a prime ideal $\mathfrak{p}\subset  R$ is an {\it associated prime} of $I$ if there exists
 an element $u$ in $R$ such that $\mathfrak{p}=(I:_R u)$, where $(I:_R u)=\{r\in R \mid  ru\in I\}$. The  {\it set of associated primes} of $I$, denoted by  $\mathrm{Ass}_R(R/I)$, is the set of all prime ideals associated to  $I$. In  \cite{BR}, Brodmann proved  that the sequence $\{\mathrm{Ass}_R(R/I^n)\}_{n \geq 1}$ of associated prime ideals is stationary  for large $n$, i.e., there exists a positive integer $n_0$ such that $\mathrm{Ass}_R(R/I^n)=\mathrm{Ass}_R(R/I^{n_0})$ for all $n\geq n_0$. The  minimal such $n_0$ is called the {\it index of stability}   of  $I$ and $\mathrm{Ass}_R(R/I^{n_0})$ is called the {\it stable set } 
 of associated prime ideals of  $I$, which is denoted by $\mathrm{Ass}^{\infty }(I).$

 Section \ref{Main Section 1} is concerned with the study of the notion of the strong persistence index of an ideal.
 Assume that $I$ is an ideal in   a commutative  Noetherian ring $R$. We say that a positive integer  $\ell_0$ is the \textit{strong persistence index} of $I$ if $\ell_0$ is the smallest integer such that $(I^{\ell+1}:_RI)=I^{\ell}$ for all  $\ell\geq \ell_0$.  The  ideal $I$ is said to have the  {\it strong persistence property} if  $\ell_0=1$.  The concept    of the strong persistence property of ideals has been investigated  over the last decade, for more details and information consult  \cite{ANR2, BNT, KNT, N2, NKA, RNA, RT}.
 In this section, by means of Example \ref{Ex.Ass.PandN} and  Lemmata  \ref{Lem.SPP.Index} and \ref{Mat-Local},  we show that  there exist infinitely many  monomial ideals  in   $R=K[x_1, \ldots, x_n]$, where $n\geq 3$,   that  have  the  same  strong persistence index $\ell_0\geq 2$, cf.   Theorem \ref{Th.SPP.Index}.  
 We conclude this section with some results on the strong persistence index of ideals in commutative Noetherian rings.
 
 Let \( I \) be a monomial ideal in the polynomial ring  \( R = K[x_1, \ldots, x_n] \). Suppose that \( n_0 \) denotes the 
index of stability of \( I \), and \( \ell_0 \) denotes the strong persistence index of \( I \). It is natural to ask whether these 
indices can be compared. In Remark \ref{Compare} we present two examples showing that, in the first example, \( n_0 < \ell_0 \), while in the 
second example, \( n_0 > \ell_0 \). In other words, these two indices are not comparable in general.

  Section   \ref{Main Section 2} is devoted to exploring the second topic of our study. We say that an ideal $I$  in a commutative  Noetherian $R$ 
  has the  phenomenon  of \textit{fluctuation in colon powers}
 if there exist  positive integers $a<b<c$ such that at least one of the following cases occurs:
\begin{itemize}
\item[(i)]  $(I^{a}:I) = I^{a-1}, \quad (I^{b}:I) \neq I^{b-1}, \quad \text{but} \quad (I^{c}:I) = I^{c-1}.$
\item[(ii)] $(I^{a}:I) \neq I^{a-1}, \quad (I^{b}:I) = I^{b-1}, \quad \text{but} \quad (I^{c}:I) \neq I^{c-1}.$
\end{itemize}
    A natural question is whether there exist infinitely many monomial ideals in $R=K[x_1, \ldots, x_n]$, where $n\geq 2$,  whose colon powers exhibit fluctuation. This question is answered in the affirmative in Theorem \ref{Th.Fluctuation}.  
  
  \bigskip
Throughout this paper,  $\mathcal{G}(I)$ denotes the unique minimal set of monomial generators of a monomial ideal $I\subset R=K[x_1, \ldots, x_n]$, where 
$R=K[x_1, \ldots, x_n]$ is  a polynomial  ring over a field $K$. Furthermore, the {\em support} of a monomial $u\in R$, denoted by $\mathrm{supp}(u)$, is the set of variables that divide $u$.  Moreover, for a monomial ideal $I$, we set $\mathrm{supp}(I)=\bigcup_{u \in \mathcal{G}(I)}\mathrm{supp}(u)$.

%%%%%%%%%%%%%%%%%%%%%%%%%%%%%%%%%%%%%%%%%%%%%%%%%%%%%%%%%%%%
%%%%%%%%%%%%%%%%%%%%%%%%%%%%%%%%%%%%%%%%%%%%%%%%%%%%%%%%%%%%
%%%%%%%%%%%%%%%%%%%%%%%%%%%%%%%%%%%%%%%%%%%%%%%%%%%%%%%%%%%%
%%%%%%%%%%%%%%%%%%%%%%%%%%%%%%%%%%%%%%%%%%%%%%%%%%%%%%%%%%%%

\section{On the strong persistence index  of monomial ideals} \label{Main Section 1}

 First, recall that for an ideal $I$  in  a commutative  Noetherian ring $R$, we  say that a positive integer  $\ell_0$ is the \textit{strong persistence index} of $I$ if $\ell_0$ is the smallest integer such that $(I^{\ell+1}:_RI)=I^{\ell}$ for all  $\ell\geq \ell_0$. Suppose now that $R=K[x_1, \ldots, x_n]$ is a polynomial ring 
 over  a field $K$.  It is natural to ask whether there exists a monomial ideal $I\subset R$ with $\mathrm{supp}(I)=\{x_1, \ldots, x_n\}$  that has the 
 strong persistence index $\ell_{0}$. We provide an affirmative  answer to this question in Theorem \ref{Th.SPP.Index}.  To  this  end, we begin  with 
  Lemma \ref{Lem.Ass.PandN} and     Example \ref{Ex.Ass.PandN}. Next,   we focus on   Lemmata  \ref{Lem.SPP.Index} and \ref{Mat-Local}.

 %%%%%%%%%%%%%%%%%%%%%%%%%%%%%%%%%%%%%%%%%%%%%%%%%%%%%%%%%%
 
  Before establishing Lemma \ref{Lem.Ass.PandN}, we recall the following proposition.
 
 \begin{proposition}\label{Pro.supp} (\cite[Proposition 4.2]{NKRT})
Let  $I$ be   a monomial ideal in $R=K[x_1, \ldots, x_n]$ over a field $K$ with $\mathcal{G}(I)=\{u_1, \ldots, u_m\}$ and $\mathrm{Ass}_R(R/I)=\{\mathfrak{p}_1, \ldots, \mathfrak{p}_s\}$. Then, the following statements hold. 
\begin{itemize}
\item[(i)] If $x_i|u_t$ for some $i$ with  $1\leq i \leq n$, and  
for some $t$ with  $1\leq t \leq m$, then there exists $j$ with  $1\leq j \leq s,$ 
such that $x_i\in \mathfrak{p}_j$. 
\item[(ii)] If $x_i\in \mathfrak{p}_j$ for some $i$ with  $1\leq i \leq n$, and  for some $j$ with   $1\leq j \leq s$, then  there exists $t$ with  
 $1\leq t \leq m$, such that  $x_i|u_t$. 
\end{itemize} 
Especially, $\bigcup_{j=1}^s \mathrm{supp}(\mathfrak{p}_j)=\bigcup_{t=1}^m \mathrm{supp}(u_t)$.
\end{proposition}

 %%%%%%%%%%%%%%%%%%%%%%%%%%%%%%%%%%%%%%%%%%%%%%%%%%%%%%%%%%

 \begin{lemma} \label{Lem.Ass.PandN}
Let $I\subset R=K[x,y,z]$ be a monomial ideal, $\mathfrak{m}=(x,y,z)$, and $I=Q \cap L$ be a minimal primary decomposition of $I$ such that 
$\sqrt{Q} \in \{(x,y), (x,z), (y,z)\}$ and  $\sqrt{L}=\mathfrak{m}$. Also, let  $\mathfrak{m}\notin \mathrm{Ass}(R/I^t)$ for some $t\geq 2$.  Then $\mathfrak{m}\notin \mathrm{Ass}(R/I^\alpha)$ for all $\alpha \geq t$. 
\end{lemma}

\begin{proof}
It is enough to show that $\mathfrak{m}\notin \mathrm{Ass}(R/I^{t+1})$. Suppose, on the contrary, that $\mathfrak{m}\in \mathrm{Ass}(R/I^{t+1})$. 
Without loss of generality, assume that $\sqrt{Q}=(x,y)$.  According to  $\mathrm{Min}(I)=\{\sqrt{Q}\}$,  we can deduce that   $\mathrm{Ass}(R/I^{t+1})=\{\sqrt{Q}, \mathfrak{m}\}$. It is well-known that  $I^{t+1}$ is generated by 
$$\{uv \mid u\in \mathcal{G}(I) \text { and } v\in \mathcal{G}(I^t)\}.$$ 
 Let $f=uv$ with $u\in \mathcal{G}(I)$ and $v\in \mathcal{G}(I^t)$.  In what follows, we  prove that there exists $f_1\in I^{t+1}$ 
 with  $z\nmid f_1$ such that $f_1\mid f$.  Since  $\mathfrak{m}\notin \mathrm{Ass}(R/I^t)$ and  $\mathrm{Min}(I)=\{\sqrt{Q}\}$,  it follows that 
 $\mathrm{Ass}(R/I^{t})=\{\sqrt{Q}\}$, and by virtue of  $z\notin (x,y)$, we can conclude from Proposition \ref{Pro.supp} that $z\nmid v$. If $z\nmid u$, then  $z\nmid uv$, and so the argument is done. Hence, let $z\mid u$. Since $v\in \mathcal{G}(I^t)$,  hence there exist $g_1, \ldots, g_t \in \mathcal{G}(I)$ such that $v=g_1 \cdots g_t$. As  $ug_2 \cdots g_t \in I^t$,  we derive that there exist $m_1, \ldots, m_t \in \mathcal{G}(I)$ such that $z\nmid m_1 \cdots m_t$ and $m_1 \cdots m_t \mid ug_2 \cdots g_t$.  Set $f_1:=m_1 \cdots m_t g_1$. On account of   $f_1\mid ug_1 g_2 \cdots g_t$, $f_1 \in I^{t+1}$,   and $z\nmid f_1$, this finishes the proof. 
\end{proof}

%%%%%%%%%%%%%%%%%%%%%%%%%%%%%%%%%%%%%%%%%%%%%%%%%%%%%%%%%%%
%%%%%%%%%%%%%%%%%%%%%%%%%%%%%%%%%%%%%%%%%%%%%%%%%%%%%%%%%%%

As an application of  Lemma  \ref{Lem.Ass.PandN},  we present the following example. Notably, the construction of this example draws from \cite[Theorem 4.1]{HH2}.

\begin{example} \label{Ex.Ass.PandN}
\em{
Assume that  $I=(x^{m+3}, y^{m+3}, x^{m+2}y, xy^{m+2}, x^{m+1}y^{2}z^r)$ with $m, r\geq 1$  is  a monomial ideal in  $R=K[x,y,z]$ and 
$\mathfrak{m}=(x,y,z)$.   We show  that 
\[
\mathrm{Ass}(I^s) =
    \begin{dcases}
    \{(x,y), \mathfrak{m}\} & \text{if } 1\leq s\leq m\\
\{(x,y)\}  & \text{if } s\geq m+1. \\
                \end{dcases}
\]
We first note that  $$I=(x, y^{m+3}) \cap (x^{m+3}, y) \cap (x^{m+1}, y^{m+2}) \cap (x^{m+2}, y^2) \cap (x^{m+2}, y^{m+2}, z^r).$$
This yields that $\mathrm{Ass}(I)=\{(x,y), (x,y,z)\}$. In view of  $\mathrm{Min}(I)=\{(x,y)\}$, 
we get  $ \{(x,y)\} \subseteq   \mathrm{Ass}(I^t) \subseteq \{(x,y), \mathfrak{m}\}$ for all $t\geq 1$.

 In what follows, we assume  $1\leq s\leq m$, and our aim is to show that $\mathfrak{m} \in \mathrm{Ass}(I^s)$. To  do this, put $h:=x^{m+1}y^{s(m+2)-1}z^{r-1}$. Since $s\leq m$,  we obtain the following 
  \begin{enumerate}
  \item $xh= x^{m+2}y^{s(m+2)-1}z^{r-1}$. Due to   $(x^{m+2}y)(y^{m+3})^{s-1}\mid xh$, this leads to  $xh\in I^s$. 
  \item $yh=x^{m+1}y^{s(m+2)}z^{r-1}$. Because  $(xy^{m+2})^s \mid yh$, we have  $yh\in I^s$. 
  \item  $zh=x^{m+1}y^{s(m+2)-1}z^r$. Since $(y^{m+3})^{s-1}(x^{m+1}y^2z^r) \mid zh$, this gives that  $zh\in I^s$. 
  \end{enumerate}

Consequently, we get $(x,y,z) \subseteq (I^s:h)$ for all $1\leq s \leq m$. Here, we show that $h\notin I^s$. On the contrary, assume that $h\in I^s$. Hence, there exist nonnegative integers $a,b,c, d,$ and $e$ with $a+b+c+d+e=s$ such that 
$$(x^{m+3})^a (y^{m+3})^b (x^{m+2}y)^c (xy^{m+2})^d (x^{m+1}y^2z^r)^e \mid h.$$
This gives the following conditions:
\begin{itemize}
\item[(i)] $(m+3)a+ (m+2)c + d + (m+1)e \leq m+1$;
\item[(ii)] $(m+3)b+c+(m+2)d+2e \leq s(m+2)-1$;
\item[(iii)] $re \leq r-1$;
\item[(iv)] $a+b+c+d+e=s$, where $a,b,c,d,e \geq 0$. 
\end{itemize}
It follows from (iii) that $e=0$, and by (iv), we get $a+b+c+d=s$. If $a\geq 1$ or $c\geq 1$, this contradicts (i). Therefore, we must have $a=0$ and $c=0$. 
Since $a+b+c+d=s$, we gain $d=s-b$. Now, the left hand side of (ii) is $ms+2s+b$, while the right hand side of (ii) is $ms+2s-1$. Because $b \geq 0$, this 
leads to a contradiction. Accordingly, $h\notin I^s$, and so   $\mathfrak{m} \in \mathrm{Ass}(I^s)$  for all $1\leq s \leq m$. 

To complete our argument, we need to verify that $\mathfrak{m} \notin \mathrm{Ass}(I^s)$  for all $s\geq m+1$. From now on, our strategy is to rely on  
Lemma \ref{Lem.Ass.PandN}. For this purpose, it is sufficient for us  to show $\mathfrak{m} \notin \mathrm{Ass}(I^{m+1})$. 
We claim  if $v\in I^{m+1}$, then there exists  $f\in I^{m+1}$ with  $z \nmid f$ such that $f\mid v$.  Pick 
\begin{equation}
v=(x^{m+3})^a (y^{m+3})^b (x^{m+2}y)^c (xy^{m+2})^d (x^{m+1}y^2z^r)^e \in I^{m+1}, \tag{$\star$}
\end{equation}
  with $a\geq 0,$ $b\geq 0,$  $c\geq 0,$ 
$d \geq0$, $e\geq 1$, and $a+b+c+d+e=m+1$. Because  $(x^{m+1}y^2z^r)^{m+1}=(x^{m+2}y)^m (xy^{m+2})z^{r(m+1)}$, 
this permits us to assume that  $1\leq e \leq m$.  We first show  the claim holds when $e=1$. Let $m=1$. 
 Then   we have $I=(x^4, y^4, x^3y, xy^3, x^2y^2z^r)$, and we may meet  the following cases:
  \begin{itemize}
\item[$\bullet$] $y^4(x^3y) \mid (xy^3)(x^2y^2z^r)$, where $y^4(x^3y)\in \mathcal{G}(I^2)$.
\item[$\bullet$]  $x^4(xy^3) \mid  (x^3y)(x^2y^2z^r)$, where $x^4(xy^3) \in \mathcal{G}(I^2)$.
\item[$\bullet$]  $(xy^3)^2 \mid y^4(x^2y^2z^r)$, where $(xy^3)^2 \in \mathcal{G}(I^2)$.
\item[$\bullet$]   $(x^3y)^2 \mid x^4(x^2y^2z^r)$, where $(x^3y)^2\in \mathcal{G}(I^2)$.
\end{itemize}
 Now, let $m\geq 2$.   One can easily see the following statements:

  \begin{itemize}
\item[(1)]  $x^{m+3}(x^{m+2}y)^2 \mid (x^{m+3})^{2} (x^{m+1}y^2z^r)$. 
\item[(2)]  $y^{m+3}(x^{m+2}y)^2 \mid x^{m+3} y^{m+3} (x^{m+1}y^2z^r)$. 
\item[(3)]   $x^{m+3} (y^{m+3})^2 \mid (xy^{m+2})^2(x^{m+1}y^2z^r)$.
\item[(4)]   $(x^{m+2}y)^3 \mid x^{m+3}(x^{m+2}y)(x^{m+1}y^2z^r)$.
\item[(5)]  $x^{m+3} y^{m+3}(x^{m+2}y)  \mid  x^{m+3}(xy^{m+2})(x^{m+1}y^2z^r)$.
\item[(6)]  $y^{m+3}(x^{m+2}y)^2 \mid (x^{m+2}y)(xy^{m+2})(x^{m+1}y^2z^r)$.
\item[(7)]  $(y^{m+3})^2 (x^{m+2}y) \mid y^{m+3}(xy^{m+2})(x^{m+1}y^2z^r)$.
\item[(8)]  $(x^{m+3})^{m-\alpha} (xy^{m+2})^{\alpha +1} \mid (y^{m+3})^{\alpha} (x^{m+2}y)^{m-\alpha} (x^{m+1}y^2z^r)$,  where \\
  $0\leq \alpha \leq m$. 
\end{itemize}
 By considering the statements above,  we can  deduce that the claim is true when $e=1$ and $m\geq 2$. 
 Now, we use induction on $m$. Let $m=1$. Then $e=1$.  The argument above states that our claim holds in case $m=1$.  Now, suppose, inductively, that $m\geq 2$ and that the result has been proved for all values  less than $m$. Write $v=v_1(x^{m+1}y^2z^r)$. It follows from ($\star$) that 
  $v_1 \in I^m$, and so  we can conclude from  the inductive hypothesis that there exists a monomial $f_1\in I^{m}$ with  $z \nmid f_1$ such that $f_1\mid v_1$.    Furthermore,  in view of  the  statements (1)--(8), we can derive that  there exists a monomial $f_2\in I^{m+1}$ with  $z \nmid f_2$ such that $f_2 \mid f_1(x^{m+1}y^2z^r)$, and hence $f_2 \mid v$.   This completes the inductive step, and thus the claim  is  proved by induction.  
  
  Since $\mathfrak{m} \notin \mathrm{Ass}(I^{m+1})$,  based on  Lemma \ref{Lem.Ass.PandN}, this gives rise to   
  $\mathfrak{m} \notin \mathrm{Ass}(I^s)$  for all $s\geq m+1$, as required. 
}
\end{example}

 %%%%%%%%%%%%%%%%%%%%%%%%%%%%%%%%%%%%%%%%%%%%%%%%%%%%%%%%%
 %%%%%%%%%%%%%%%%%%%%%%%%%%%%%%%%%%%%%%%%%%%%%%%%%%%%%%%%%
 
 To understand  Lemma \ref{Lem.SPP.Index}, we need to review the definitions of expansion and weighting operations. 
 We begin by recalling   the definition of the expansion operation  on monomial ideals from  \cite{BH}. \par 
Suppose  $R = K[x_1, \ldots , x_n]$ is  the polynomial ring over a field $K$ in the variables $x_1, \ldots , x_n$. Fix an ordered $n$-tuple $(i_1, \ldots , i_n)$ of positive integers, and take the polynomial ring $R^{(i_1,\ldots ,i_n)}$ over $K$ in the variables $$x_{11}, \ldots , x_{1i_1} , x_{21}, \ldots , x_{2i_2} , \ldots , x_{n1}, \ldots , x_{ni_n}.$$
Assume  $\mathfrak{p}_j$ is  the monomial prime ideal $(x_{j1}, x_{j2}, \ldots , x_{ji_j}) \subseteq R^{(i_1,\ldots,i_n)}$ for all $j=1,\ldots,n$. Attached to each monomial ideal $I \subset R$ a set of monomial generators $\{\mathbf {x}^{\mathbf{ a}_1} , \ldots , {\mathbf x}^{{\mathbf a}_m}\}$, where ${\mathbf x}^{\mathbf a_i}={x_1}^{a_i(1)}\cdots {x_n}^{a_i(n)}$ and $a_i(j)$ denotes the $j$th component of the vector ${\mathbf a}_i=(a_i(1),\ldots,a_i(n))$ for all $i=1,\ldots,m$. We define the {\it expansion of I with respect to the n-tuple $(i_1, \ldots, i_n)$}, denoted by $I^{(i_1,\ldots,i_n)}$, to be the monomial ideal
$$I^{(i_1,\ldots,i_n)} = \sum_{i=1}^m \mathfrak{p}_1^{a_i(1)}\cdots \mathfrak{p}_n^{a_i(n)}\subseteq R^{(i_1,\ldots,i_n)}.$$
We simply write $R^*$ and $I^*$,
respectively, rather than $R^{(i_1,\ldots,i_n)}$ and $I^{(i_1,\ldots,i_n)}$.\par
To see  an  example, consider    $R = K[x_1, x_2, x_3]$ and the ordered $3$-tuple $(3,1,2)$. We thus have 
 $\mathfrak{p}_1 = (x_{11},x_{12}, x_{13})$, $\mathfrak{p}_2 = (x_{21})$, and $\mathfrak{p}_3 = (x_{31}, x_{32})$. This yields  that 
for the monomial ideal $I = (x_1^3, x_2x_3^2,  x_1 x_3)$, the ideal $I^* \subseteq K[x_{11}, x_{12}, x_{13}, x_{21}, x_{31}, x_{32}]$ is $\mathfrak{p}_1^3+\mathfrak{p}_2\mathfrak{p}^2_3+\mathfrak{p}_1\mathfrak{p}_3$, namely
\begin{align*}
I^*=(&x_{11}^3, x_{12}^3, x_{13}^3, x_{11}^2x_{12}, x_{11}^2x_{13}, x_{12}^2 x_{11}, x_{12}^2 x_{13}, x_{13}^2 x_{11}, 
x_{13}^2 x_{12}, x_{11} x_{12} x_{13}, \\
& x_{21} x_{31}^2, x_{21} x_{32}^2, x_{21} x_{31}x_{32}, x_{11} x_{31},  x_{11}x_{32}, x_{12}x_{31}, x_{12} x_{32}, x_{13}x_{31}, x_{13}x_{32}). 
\end{align*}

%+++++++++++++++++++++++++++++++++++++++++++++++++++++++++++++++++++++++++++++++
%+++++++++++++++++++++++++++++++++++++++++++++++++++++++++++++++++++++++++++++++

\bigskip
 To show Lemma \ref{Lem.SPP.Index}, we need the  following  lemma, which we state here for convenience.

\begin{lemma} (\cite [Lemma 1.1]{BH}) \label{Lem.Bayati.Expansion}
Let $I$ and $J$ be monomial ideals in a polynomial ring $S$. Then
\begin{itemize}
\item[(i)] $f \in I^*$ if and only if $\pi(f)\in I$, for all $f\in S^*$;
\item[(ii)] $(I + J)^* = I^* + J^*$;
\item[(iii)] $(IJ)^* = I^*J^*$;
\item[(iv)] $(I \cap J)^* = I^*\cap J^*$;
\item[(v)] $(I : J)^* = (I^* : J^*)$;
\item[(vi)] $\sqrt{I^*} = (\sqrt{I})^*$;
\item[(vii)] If the monomial ideal $Q$ is $\mathfrak{p}$-primary, then $Q^*$ is
$\mathfrak{p}^*$-primary.
\end{itemize}
\end{lemma}

%+++++++++++++++++++++++++++++++++++++++++++++++++++++++++++++++++++++++++++++++
%+++++++++++++++++++++++++++++++++++++++++++++++++++++++++++++++++++++++++++++++

In the sequel, we recall the definition of the weighting operation, beginning with the following definition.

\begin{definition}
\em{
A \textit{weight} over a polynomial ring $R=K[x_1, \ldots, x_n]$ is a function $W: \{x_1, \ldots, x_n\} \rightarrow \mathbb{N}$ such that  $w_i=W(x_i)$ is called the {\it weight} of the variable $x_i$. For a monomial ideal $I\subset R $ and a weight $W$, we define the \textit{weighted ideal}, denoted 
by $I_W$, to be the ideal generated by $\{h(u) : u\in \mathcal{G}(I) \}$, where $h$ is the unique homomorphism $h: R \rightarrow R$ given by 
$h(x_i)= x_i^{w_i}$. 
}
\end{definition}

%++++++++++++++++++++++++++++++++++++++++++++++++++++++++++++++++++++++++++
%++++++++++++++++++++++++++++++++++++++++++++++++++++++++++++++++++++++++++

As  an example, assume   $R=K[x_1, x_2, x_3, x_4, x_5]$ and the monomial ideal  $I=(x_1^2x_3x_4^5, x_2^4x_4^3x_5^2, x_1x_3^2, x_4x_5^3)$ in $R$. 
Also,  let $W : \{x_1, x_2, x_3, x_4, x_5\} \rightarrow \mathbb{N}$ be a weight over $R$ with $W(x_1)=1$, $W(x_2)=4$, $W(x_3)=2$, $W(x_4)=3$, and $W(x_5)=2$. It is easy to  check that  the weighted ideal $I_W$ is given by 
$$I_W=(x_1^2x_3^2x_4^{15}, x_2^{16}x_4^9x_5^4, x_1x_3^4, x_4^3x_5^6).$$

\bigskip
The next auxiliary result will be used in the proof of Lemma \ref{Lem.SPP.Index}, which we include it  here to facilitate convenient access and easy reference.

\begin{lemma}\label{LEM. Weighted} (\cite[Lemma 3.5]{SN})
Let $I$ and $J$ be two monomial ideals of a polynomial ring $R=K[x_1, \ldots, x_n]$, and $W$ a weight over $R$. Then the following statements hold. 
\begin {itemize}
\item[(i)] $(I+J)_W= I_W + J_W$;
\item[(ii)] $(IJ)_W= I_W J_W$;
\item[(iii)] $(I\cap J)_W= I_W \cap J_W$;
\item[(iv)] $(I :_RJ)_W= (I_W :_R J_W)$.
\end{itemize}
\end{lemma}

%++++++++++++++++++++++++++++++++++++++++++++++++++++++++++++++++++++++++++
%++++++++++++++++++++++++++++++++++++++++++++++++++++++++++++++++++++++++++

 We are now in a position to state and prove Lemma \ref{Lem.SPP.Index},  which is essential for formulating Theorem \ref{Th.SPP.Index}.

\begin{lemma}\label{Lem.SPP.Index} 
Let  $I \subset R=K[x_1, \ldots, x_n]$ be a monomial ideal  and $\ell_{0}\geq 1$. Then the following statements hold.
\begin{itemize}
\item[(i)] $I$ has the strong persistence index $\ell_0$  if and only if $I^*$ has the  strong persistence index $\ell_0$, 
  where $I^*$ denotes the  expansion of $I$.  
\item[(ii)] $I$ has the strong persistence index $\ell_0$  if and only if $I_W$ has the  strong persistence index $\ell_0$, 
  where $I_W$ denotes the  weighted ideal.    
\end{itemize}
\end{lemma}

\begin{proof}
(i) We first observe that for any two monomial ideals $F$ and $H$, Lemma \ref{Lem.Bayati.Expansion}(i) implies that $H=F$ if and only if $H^*=F^*$. 
Now, we assume that $I$ has the strong persistence index $\ell_0$. This implies that $\ell_0$ is the smallest integer such that $(I^{\ell+1}:_RI)=I^{\ell}$ for all  $\ell\geq \ell_0$. According to Lemma \ref{Lem.Bayati.Expansion}(v), we have $(I^{\ell+1} :_R I)^* = ((I^{\ell+1})^*:_{R^*} I^*)$, and hence 
 $((I^{\ell+1})^*:_{R^*} I^*)=(I^{\ell})^{*}$  for all  $\ell\geq \ell_0$. It follows from Lemma \ref{Lem.Bayati.Expansion}(iii) that 
$(I^{\ell+1})^*=(I^*)^{\ell+1}$ and $(I^{\ell})^*=(I^*)^{\ell}$.  We therefore can deduce that   $((I^*)^{\ell+1}:_{R^*} I^*)=(I^*)^{\ell}$  for all  $\ell\geq \ell_0$. Suppose, on the contrary, that $((I^*)^{\ell'+1}:_{R^*} I^*)=(I^*)^{\ell'}$  for some   $\ell'\leq \ell_0-1$.
 One can immediately derive from parts (iii) and (v) of  Lemma \ref{Lem.Bayati.Expansion}   that $(I^{\ell'+1}:_RI)=I^{\ell'}$, which contradicts the minimality of $\ell_0$. 
Consequently, we get $I^*$ has the  strong persistence index $\ell_0$. 
 By a similar argument,  we can show  the converse.

(ii)   Note that for any two monomial ideals $F$ and  $H$, we have $H=F$ if and only if $H_W=F_W$. Now, the claim can be established by combining Lemma \ref{LEM. Weighted} and mimicking the proof of part (i).
\end{proof}

%++++++++++++++++++++++++++++++++++++++++++++++++++++++++++++++++++++++++++
%++++++++++++++++++++++++++++++++++++++++++++++++++++++++++++++++++++++++++

To formulate Theorem \ref{Th.SPP.Index}, we have to employ Lemma \ref{Mat-Local}.  Before stating it, we recall the definition of the  monomial localization. Let  $I\subset R=K[x_1, \ldots, x_n]$ be  a monomial ideal, where $K$ is a field. We  denote by $V^*(I)$ the set of monomial prime ideals containing $I$. Let $\mathfrak{p}=(x_{i_1}, \ldots, x_{i_r})$ be a monomial prime ideal. Then the  {\it monomial localization} of $I$ with respect to $\mathfrak{p}$, denoted by $I(\mathfrak{p})$, is the ideal in the polynomial ring $R(\mathfrak{p})=K[x_{i_1}, \ldots, x_{i_r}]$  which is obtained from $I$ by using the $K$-algebra homomorphism $R\rightarrow R(\mathfrak{p})$  with $x_j\mapsto 1$  for all $x_j\notin \{x_{i_1}, \ldots, x_{i_r}\}$. 
 We also  need to recall the following auxiliary  results, which are  necessary for us to show  Lemma \ref{Mat-Local}. We insert them here for ease of reference.

%%%%%%%%%%%%%%%%%%%%%%%%%%%%%%%%%%%%%%%%%%%%%%%%%%%%%%%%%%%
%%%%%%%%%%%%%%%%%%%%%%%%%%%%%%%%%%%%%%%%%%%%%%%%%%%%%%%%%%%

\begin{fact}\label{fact3} (\cite[Exercise 9.42]{sharp}) 
Let 
$$0\longrightarrow L{\longrightarrow} M {\longrightarrow} N  \longrightarrow  0,$$
be a short exact sequence of modules and homomorphisms over the commutative Noetherian ring $R$. Then, 
$\mathrm{Ass}(L) \subseteq  \mathrm{Ass}(M) \subseteq 
 \mathrm{Ass}(L) \cup  \mathrm{Ass}(N).$ 
\end{fact}

%%%%%%%%%%%%%%%%%%%%%%%%%%%%%%%%%%%%%%%%%%%%%%%%%%%%%%%%%%%
%%%%%%%%%%%%%%%%%%%%%%%%%%%%%%%%%%%%%%%%%%%%%%%%%%%%%%%%%%%

  \begin{lemma}\label{LEM.Localization} (\cite[Lemma 3.13]{SN})
Let $I$ and $J$ be two monomial ideals   in $R=K[x_1, \ldots, x_n]$, and $\mathfrak{p}$ be a monomial prime ideal of $R$. Then the following 
statements hold.
\begin {itemize}
\item[(i)]  $(I+J)(\mathfrak{p})= I(\mathfrak{p}) + J(\mathfrak{p})$;
\item[(ii)] $(IJ)(\mathfrak{p})= I(\mathfrak{p})  J(\mathfrak{p})$;
\item[(iii)] $(I\cap J)(\mathfrak{p})= I(\mathfrak{p}) \cap  J(\mathfrak{p})$;
\item[(iv)] $(I :_RJ)(\mathfrak{p})= (I(\mathfrak{p}) :_{R(\mathfrak{p})}  J(\mathfrak{p}))$;
\item[(v)]  If $Q$ is a $\mathfrak{q}$-primary monomial ideal in $R$ with $ \mathfrak{q}\subseteq \mathfrak{p}$, then $Q(\mathfrak{p})$ is a $\mathfrak{q}$-primary monomial ideal in $R(\mathfrak{p})$.
\end{itemize}
 \end{lemma}

%%%%%%%%%%%%%%%%%%%%%%%%%%%%%%%%%%%%%%%%%%%%%%%%%%%%%
%%%%%%%%%%%%%%%%%%%%%%%%%%%%%%%%%%%%%%%%%%%%%%%%%%%%%

\begin{lemma} \label{Mat-Local}
 Let $I$ and $J$ be two monomial  ideals in $R=K[x_1, \ldots, x_n]$. If    $J(\mathfrak{p})\subseteq I(\mathfrak{p})$ 
 for  every  $\mathfrak{p}\in \mathrm{Ass}_R(R/I)$, then $J\subseteq I$. 
  \end{lemma}

 \begin{proof}
 Assume that   $J(\mathfrak{p})\subseteq I(\mathfrak{p})$   for  every   $\mathfrak{p}\in \mathrm{Ass}_R(R/I)$. We first demonstrate  that $I(\mathfrak{p})R_{\mathfrak{p}}=I_{\mathfrak{p}}$, where $I_{\mathfrak{p}}$ means the localization of $I$ at $\mathfrak{p}$. To establish this claim, 
 let  $I=(u_1, \ldots, u_t)$. Then, Lemma \ref{LEM.Localization}(i) yields that  $I({\mathfrak{p}})=(u_1({\mathfrak{p}}), \ldots, u_t({\mathfrak{p}}))$, 
 and so    $I({\mathfrak{p}})R_{\mathfrak{p}}=(u_1({\mathfrak{p}})R_{\mathfrak{p}}, \ldots, u_t({\mathfrak{p}})R_{\mathfrak{p}})$. 
 It is also well-known that $I_{\mathfrak{p}}=(({u_1})_{\mathfrak{p}}, \ldots, ({u_t})_{\mathfrak{p}})$. Fix $1\leq i \leq t$. We want to 
 show that $u_i({\mathfrak{p}})R_{\mathfrak{p}}=({u_i})_{\mathfrak{p}}$. It is easy  to check  (for example see \cite[Lemma 5.25]{sharp}),  
 $$({u_i})_{\mathfrak{p}}=\{\frac{a}{r} \mid a\in (u_i) \text{~and~} r\notin \mathfrak{p}\} \text{~and~} {u_i}(\mathfrak{p})R_{\mathfrak{p}}
 =\{\frac{b}{h} \mid b\in  (u_i(\mathfrak{p})) \text{~and~} h\notin \mathfrak{p}\} .$$
 We can write $u_i=fg$ with $\mathrm{supp}(f)\subseteq \mathrm{supp}(\mathfrak{p})$ and $\mathrm{supp}(g)\subseteq 
 \{x_1, \ldots, x_n\} \setminus  \mathrm{supp}(\mathfrak{p})$. This yields that $u_i(\mathfrak{p})=f$. Now, let $a/r \in  ({u_i})_{\mathfrak{p}}$, where $a\in (u_i)$ and   $r\notin \mathfrak{p}$.  Then $a=fgw$ for some monomial $w$, and so $a\in (u_i(\mathfrak{p}))$.
  We thus obtain $a/r \in  {u_i}(\mathfrak{p})R_{\mathfrak{p}}$, and hence $({u_i})_{\mathfrak{p}} \subseteq {u_i}(\mathfrak{p})R_{\mathfrak{p}}$. 
  Conversely, choose $b/h\in {u_i}(\mathfrak{p})R_{\mathfrak{p}}$, where $b\in  (u_i(\mathfrak{p}))$ and $h\notin \mathfrak{p}.$ 
  Then $b=fw'$ for some monomial $w'$, and so $bg=fgw'\in ({u_i})$. Since $\mathrm{supp}(g) \cap \mathrm{supp}(\mathfrak{p})=\emptyset$, this gives that 
  $hg\notin \mathfrak{p}$. Therefore, we can deduce that  $b/h=(bg)/(hg)\in ({u_i})_{\mathfrak{p}}.$
  This yields  that  ${u_i}(\mathfrak{p})R_{\mathfrak{p}} \subseteq ({u_i})_{\mathfrak{p}}$. Consequently, we have $I(\mathfrak{p})R_{\mathfrak{p}}=I_{\mathfrak{p}}$.  Now, suppose, on the contrary, that $J\nsubseteq  I$. 
  This implies that $I+J \neq I$, and so $(I+J)/I\neq 0$. Hence, we get 
 $\mathrm{Ass}_R((I+J)/I)\neq \emptyset$.   Pick an arbitrary element $\mathfrak{q}\in \mathrm{Ass}_R((I+J)/I)$. Hence, 
 $\mathfrak{q}R_{\mathfrak{q}}\in \mathrm{Ass}_{R_{\mathfrak{q}}}((I+J)/I)_{\mathfrak{q}}$.   Here, we   consider the following short exact sequence
 \begin{equation}\label{sequence1}
0\longrightarrow \frac{I+J}{I} \longrightarrow  \frac{R}{I}  \longrightarrow \frac{R}{I+J} \longrightarrow  0.
\end{equation} 
From  (\ref{sequence1}) and Fact \ref{fact3}, we obtain $\mathrm{Ass}_R((I+J)/I) \subseteq \mathrm{Ass}_R(R/I)$, and so  $\mathfrak{q}\in \mathrm{Ass}_R(R/I)$. 
It follows  from our assumption that $J(\mathfrak{q})\subseteq I(\mathfrak{q})$, and hence  $J(\mathfrak{q})R_{\mathfrak{p}}\subseteq I(\mathfrak{q})R_{\mathfrak{p}}$. Based on the 
 above discussion, we can deduce that   $J_{\mathfrak{q}}  \subseteq I_{\mathfrak{q}}$.
  This leads to $I_{\mathfrak{q}} + J_{\mathfrak{q}}=  I_{\mathfrak{q}}$, and thus $(I_{\mathfrak{q}} + J_{\mathfrak{q}})/ I_{\mathfrak{q}}=0$.  
    Since $((I+J)/I)_{\mathfrak{q}}=(I_{\mathfrak{q}} + J_{\mathfrak{q}})/ I_{\mathfrak{q}}$, we obtain  $((I+J)/I)_{\mathfrak{q}}=0$, which 
    contradicts the fact that  $\mathfrak{q}R_{\mathfrak{q}}\in \mathrm{Ass}_{R_{\mathfrak{q}}}((I+J)/I)_{\mathfrak{q}}$. Therefore, we must have $J\subseteq I$, and the proof is complete.  
 \end{proof}

The following theorem  ensures that there exist infinitely many  monomial ideals  with the same strong persistence index.

\begin{theorem} \label{Th.SPP.Index}
Let $R=K[x_1, \ldots, x_n]$ with $n\geq 3$ be  a polynomial ring in $n$  variables  with coefficients in a  field $K$, and $\ell_{0}\geq 2$. 
Then there exist infinitely many  monomial ideals  in $R$  that  have  the  same strong persistence index $\ell_0$. 
\end{theorem}

\begin{proof}
 Our strategy for verifying this claim is to split the proof into two parts. 
First, we show that there exists some monomial $u \in (L^s : L) \setminus L^{s-1}$  for  $2 \le s \le m+1,$ and then we prove that
 $(L^s : L) = L^{s-1}$ for  $s \ge m+2$ via monomial localization.

Fix $r\geq 1$, and let  $L:= (x_1^{m+3}, x_2^{m+3}, x_1^{m+2}x_2, x_1x_2^{m+2}, x_1^{m+1}x_2^{2}x_3^r)$ in $S=K[x_1, x_2, x_3]$, where 
$m=\ell_0-1$.  It follows now from Example \ref{Ex.Ass.PandN}  that 
\[
\mathrm{Ass}_S(S/L^s) =
    \begin{dcases}
    \{(x_1,x_2), (x_1,x_2,x_3)\} & \text{if } 1\leq s\leq m\\
\{(x_1,x_2)\}  & \text{if } s\geq m+1. \\
                \end{dcases}
\]
Set $u:=x_1^{m+1}x_2^{(s-1)(m+1)+s-2}$. We claim that $u\in (L^s:_SL)\setminus L^{s-1}$ for all $2\leq s\leq m+1$. To show this, fix $2\leq s\leq m+1$.
 We now  consider the  following statements:
\begin{itemize}
\item[(1)]  $x_1^{m+3}u= (x_1^{m+2}x_2)^2 (x_2^{m+3})^{s-2} x_2^{m-s+1} \in L^s$. 
\item[(2)]  $x_2^{m+3}u= x_1^{m-s+1} (x_1x_2^{m+2})^{s} \in L^s$. 
\item[(3)]  $x_1^{m+2}x_2u=x_1^{m+3} (x_1x_2^{m+2})^{s-1}x_1^{m-s+1}\in L^s.$
\item[(4)]  $x_1x_2^{m+2}u=(x_1^{m+2}x_2)(x_2^{m+3})^{s-1}x_2^{m-s+1}\in L^s$.
\item[(5)]  $x_1^{m+1}x_2^{2}x_3^ru=x_1^{m+3}(x_1x_2^{m+2})^{s-1}x_1^{m-s}x_2x_3^r\in  L^s$.    
\end{itemize}
We thus have   $u\in (L^s:_SL)$. On the contrary, assume  that  $u\in L^{s-1}$. Then there exist nonnegative integers $a_i$, where $i=1, \ldots, 5$, 
 with $\sum_{i=1}^5a_i=s-1$ such that  $(x_1^{m+3})^{a_1} (x_2^{m+3})^{a_2}  (x_1^{m+2}x_2)^{a_3} (x_1x_2^{m+2})^{a_4} 
(x_1^{m+1}x_2^{2}x_3^r)^{a_5} \mid u$. Hence, we get the following statements: 
\begin{itemize}
\item[(i)]  $(m+3)a_1 + (m+2)a_3+ a_4 +(m+1)a_5  \leq m+1$. 
\item[(ii)] $(m+3)a_2 + a_3 + (m+2)a_4 + 2a_5 \leq (s-1)(m+1)+s-2$. 
\item[(iii)] $ra_5=0$. 
\end{itemize}
Since $r\geq 1$, from (iii), we obtain $a_5=0$. If $a_1>0$ or $a_3>0$, then we  get  $(m+3)a_1 + (m+2)a_3+ a_4 +(m+1)a_5 >m+1$, which contradicts (i). Hence, we must have  $a_1=0$ and $a_3=0$. Due to $\sum_{i=1}^5a_i=s-1$, this implies that  $a_2+a_4=s-1$.  Now, (ii) gives rise to  
$$(m+3)a_2  + (m+2)(s-1-a_2)   \leq (s-1)(m+1)+s-2.$$
The above inequality leads to $a_2\leq -1$, which is a contradiction. Hence, we can conclude that  $u\notin L^{s-1}$, as required.
 Consequently, $(L^s:_SL)\neq  L^{s-1}$ for all $2\leq s\leq m+1$.

 In what follows, we want to verify that  $(L^s:_SL)= L^{s-1}$    for all $s \geq m+2$. To accomplish this, fix $s\geq m+2$. 
  Since $L^{s-1} \subseteq (L^s:_SL)$,  it suffices to show that $(L^s:_SL) \subseteq L^{s-1}$. For this purpose, we rely on Lemma \ref{Mat-Local}. 
 On account of   $\mathrm{Ass}_S(S/L^s)=\{(x_1, x_2)\}$ (when $s\geq m+1$),  we only need  to verify that  
    $(L^s:_SL)(x_1,x_2) \subseteq L^{s-1}(x_1,x_2)$.  Put 
    $$J:=L(x_1,x_2)=(x_1^{m+3}, x_2^{m+3}, x_1^{m+2}x_2, x_1x_2^{m+2}, x_1^{m+1}x_2^{2}) \subset K[x_1,x_2].$$
    By parts (ii) and (iv) of Lemma \ref{LEM.Localization}, we should prove $(J^s:J) \subseteq J^{s-1}$.
     We claim     $(J^s: x_1^{m+3}) \subseteq  J^{s-1}$.   To simplify the notation, set  $f_1:=x_1^{m+3}$, $f_2:=x_2^{m+3}$, $f_3:=x_1^{m+2}x_2$, 
 $f_4:=x_1x_2^{m+2}$, and $f_5:=x_1^{m+1}x_2^{2}$.  Pick an arbitrary monomial $M$ in  $(J^s: f_1)$.  Then there exist nonnegative integers 
 $\alpha_1, \ldots, \alpha_5$ with  $\sum_{i=1}^5\alpha_i=s$, and some monomial $\ell=x_1^{r}x_2^{t}$ such that 
 $$Mf_1= \ell f_1^{\alpha_1} f_2^{\alpha_2} f_3^{\alpha_3} f_4^{\alpha_4} f_5^{\alpha_5}.$$ 
  If $\alpha_1>0$, then $M\in J^{s-1}$. Hence, from now on, we assume    $\alpha_1=0$.  If $\alpha_5 \geq m+3$, then $M\in J^{s-1}$ thanks to  $f_5^{\alpha_5}=f_1^{m+1}f_2^{2} f_5^{\alpha_5-m-3}$.   Accordingly, we can assume  $\alpha_5\leq m+2$. 
 Let $\alpha_3\geq 2$. Since  $f_3^{\alpha_3}=f_3^{\alpha_3-2}f_1 f_5,$ this implies that $M\in J^{s-1}$. 
 Thus, we can assume  $\alpha_3=0$ or $\alpha_3=1$. Let $\alpha_3=1$. Then we get the following equality 
  $$Mx_1^{m+3}= x_1^rx_2^t (x_2^{m+3})^{\alpha_2} (x_1^{m+2}x_2) (x_1x_2^{m+2})^{\alpha_4} (x_1^{m+1}x_2^{2})^{\alpha_5}.$$
 If $r\geq 1$ or  $\alpha_4>0$, then $M\in J^{s-1}$ due to $x_1f_3=x_2f_1$ and $f_3f_4=f_1f_2$.  So, we assume that $r=0$ and   $\alpha_4=0$. 
 Let  $\alpha_5=0$. This gives that 
 $Mx_1^{m+3}= x_2^t (x_2^{m+3})^{\alpha_2} (x_1^{m+2}x_2).$ Since  $\deg_{x_1}(Mx_1^{m+3})\geq m+3$ and 
 $\deg_{x_1}(x_2^t (x_2^{m+3})^{\alpha_2} (x_1^{m+2}x_2))=m+2$, this leads to a contradiction. Now, we let $\alpha_5=1$. Since $s\geq m+2$ and 
 $\alpha_2=s-2$, we get the following 
 $$M=x_2^t(x_2^{m+3})^{\alpha_2-m+1} (x_1x_2^{m+2})^{m}\in J^{s-1}.$$
 We thus let $\alpha_5\geq 2$. Then we have $Mf_1=x_2^tf_2^{\alpha_2} f_3f_5^{\alpha_5}$. This implies that 
  \begin{equation} \label{19}
   M=x^t_2(x_2^{m+3})^{\alpha_2} (x_1^{2m+1}x_2^5) (x_1^{m+1}x_2^2)^{\alpha_5-2}.
   \end{equation}
 Here, we need to consider the following subcases:
 
  \textbf{Case 1.}    $\alpha_2+1\geq m$.  It follows from (\ref{19}) and   $\alpha_2+\alpha_5=s-1$ that 
  $$M=x_2^t (x_2^{m+3})^{\alpha_2-m+1} (x_1x_2^{m+2})^m  (x_1^{m+1}x_2^2)^{\alpha_5-1}\in J^{s-1}.$$

\textbf{Case 2.}  $\alpha_2+1 < m$.  Then $\alpha_2\leq m-2$. Due to   $\alpha_2+\alpha_5=s-1$ and $s\geq m+2$, we have   $\alpha_5\geq 3$. 
 We can now deduce  from (\ref{19}),   $3\leq \alpha_5 \leq m+2$, and $\alpha_2+\alpha_5=s-1$ that 
    $$M= x_2^t(x_2^{m+3})^{\alpha_2+\alpha_5-m-1} (x_1x_2^{m+2})^{m+2-\alpha_5} 
       (x_1^{m+1}x_2^2) (x_1^{m+2} x_2)^{\alpha_5-2}\in J^{s-1}.$$ 
  Henceforth, we additionally  assume  $\alpha_3=0$. 
 Let  $\alpha_5 \geq 2$. If $\alpha_4\geq 2$,  then $M\in J^{s-1}$ because  $f_4^{2}f_5=f_1f_2^{2}$. Therefore, let   $\alpha_4 \leq 1$. 
  We first  suppose that   $\alpha_4=1$.     This implies the following  
 \begin{equation} \label{20}
   M=x_1^rx^t_2(x_2^{m+3})^{\alpha_2 +1} x_1^m x_2^3 (x_1^{m+1}x_2^2)^{\alpha_5-2}.
   \end{equation} 
  Now, we can observe  the following subcases:
   
   \textbf{Case 1.}    $\alpha_2+1\geq m$. From  (\ref{20})   and    $\alpha_2+\alpha_5=s-1$, we gain the following  
$$M=x_1^rx^t_2(x_1x_2^{m+2})^m  (x_2^{m+3})^{\alpha_2+2-m} (x_1^{m+1}x_2^2)^{\alpha_5-2} \in J^{s-1}.$$

\textbf{Case 2.}  $\alpha_2+1 < m$.   Then $\alpha_2\leq m-2$. As  $\alpha_2+\alpha_5=s-1$ and $s\geq m+2$, we get  $\alpha_5\geq 3$. 
 Because  $3\leq \alpha_5 \leq m+2$, one can rephrase  (\ref{20}) as follows:
      $$M= x_1^rx_2^t(x_2^{m+3})^{\alpha_2+\alpha_5-m-1} (x_1x_2^{m+2})^{m+3-\alpha_5} 
       (x_1^{m+2}x_2)^{\alpha_5-3} (x_1^{m+1}x_2^2)\in J^{s-1}.$$    
       Now, let $\alpha_4=0$.  Then we obtain  $Mf_1=x_1^rx^t_2 f_2^{\alpha_2}f_5^{\alpha_5}.$ This gives that     
    \begin{equation} \label{21}
   M=x_1^rx^t_2 (x_2^{m+3})^{\alpha_2} (x_1^{m-1} x_2^4) (x_1^{m+1}x_2^2)^{\alpha_5-2}.
     \end{equation}
     We can now  consider the following subcases:
   
\textbf{Case 1.}    $\alpha_2\geq m-1$. According to  (\ref{21}) and $\alpha_2+\alpha_5=s$, we derive that   
$$M=x_1^r x^t_2(x_1x_2^{m+2})^{m-1}  (x_2^{m+3})^{\alpha_2+2-m} (x_1^{m+1}x_2^2)^{\alpha_5-2} \in J^{s-1}.$$

\textbf{Case 2.}  $\alpha_2 < m-1$.  Then $\alpha_2\leq m-2$. Since $\alpha_2+\alpha_5=s$ and $s\geq m+2$, we get  $\alpha_5\geq 4$. 
 Due to $4\leq \alpha_5 \leq m+2$, we can rewrite (\ref{21}) as follows:
     $$M= x_1^rx_2^t(x_2^{m+3})^{\alpha_2+\alpha_5-m-2} (x_1x_2^{m+2})^{m+3-\alpha_5} 
        (x_1^{m+2}x_2)^{\alpha_5-4} (x_1^{m+1}x_2^2)^{2}\in J^{s-1}.$$
     Accordingly, we can let $\alpha_5=0$ or $\alpha_5=1$. Now, let $\alpha_5= 1$. If $\alpha_4 \geq 2$, then $M\in J^{s-1}$ due to $f_4^{2}f_5=f_1f_2^{2}$. Hence, we assume that $\alpha_4 \leq 1$. Let $\alpha_4=1$. If   $r\geq 1$, then $M\in J^{s-1}$ since $x_1f_4f_5=x_2f_1f_2$. Thus, let $r=0$. Because $\deg_{x_1}(Mf_1)\geq m+3$ and   $\deg_{x_1}(x_2^t f_2^{\alpha_2} f_4f_5)=m+2$, this gives a contradiction. Hence, 
  we let $\alpha_4=0$. If $r\geq 2$, then  $M\in J^{s-1}$   due to $x_1^2f_5=x_2^2f_1$. Thus, let $r\leq 1$. Thanks to  
  $\deg_{x_1}(Mf_1)\geq m+3$ and 
  $\deg_{x_1}(x_1^rx_2^t f_2^{\alpha_2}  f_5)=m+1+r\leq m+2$, this leads to  a contradiction. 
  Therefore, we assume that $\alpha_5=0$. If $\alpha_4 \geq m+4$, then $M\in J^{s-1}$ since $f_4^{\alpha_4}=f_1f_2^{m+2}f_4^{\alpha_4-m-3}$. 
  Thus,  let $\alpha_4\leq m+3$. Due to   $\alpha_1=\alpha_3=\alpha_5=0$, we obtain  $\alpha_2=s-\alpha_4$, and hence we get the following equalities 
     \begin{align*}
  \deg_{x_2}M=&\deg_{x_2}(Mx_1^{m+3})=\deg_{x_2}(\ell f_1^{\alpha_1} f_2^{\alpha_2} f_3^{\alpha_3} f_4^{\alpha_4} f_5^{\alpha_5})\\
 =& t+(m+3)\alpha_2+(m+2)\alpha_4\\
 =& t + (m+3)(s-1) +m+3-\alpha_4.
  \end{align*}
    Since $\alpha_4\leq m+3$, this yields that $(x_2^{m+3})^{s-1} \mid M$, and so $M\in J^{s-1}$. 

Hence, $(J^s: x_1^{m+3}) \subseteq  J^{s-1}$. As $J^{s-1} \subseteq (J^s:J) \subseteq (J^s: x_1^{m+3}) \subseteq  J^{s-1}$, we can conclude   
$(J^s:J)=J^{s-1}$.  From Lemma \ref{Mat-Local}, we get   $(L^s:_SL) \subseteq L^{s-1}$, and so $(L^s:_SL)= L^{s-1}$.
     Accordingly, we deduce that $L$ has the  strong persistence index $\ell_0=m+1$.
  Now, let $\mathfrak{p}_1:=(x_{i_1}, \ldots, x_{i_a})$, $\mathfrak{p}_2:=(x_{i_{a+1}}, \ldots, x_{i_b})$, and $\mathfrak{p}_3:=(x_{i_{b+1}}, \ldots, x_{i_c})$ with  $\mathrm{supp}(\mathfrak{p}_i) \cap  \mathrm{supp}(\mathfrak{p}_j)=\emptyset$ for all $1\leq i<j \leq 3$ and $\bigcup_{i=1}^{3}\mathrm{supp}(\mathfrak{p}_i)=\{x_1, \ldots, x_n\}$. 
 We can promptly deduce from Lemma  \ref{Lem.SPP.Index}(i) that  $L^*$ in $R=K[x_1, \ldots, x_n]$
 has the strong persistence index $\ell_0$. Put   $I:=(L^*)_W$ with $W(x_i)=\alpha_i$ such that $\alpha_i \geq 1$ for all $i=1, \ldots, n$. 
 It follows from  Lemma  \ref{Lem.SPP.Index}(ii) that  $I$  also   has the strong persistence index $\ell_0\geq 2$. Accordingly, we can conclude that   there exist
  infinitely many  such monomial ideals in $R$ that have the same strong persistence  index $\ell_0\geq 2$. This finishes the proof. 
  \end{proof}

%%%%%%%%%%%%%%%%%%%%%%%%%%%%%%%%%%%%%%%%%%%%%%%%%%%%%%%%%%
%%%%%%%%%%%%%%%%%%%%%%%%%%%%%%%%%%%%%%%%%%%%%%%%%%%%%%%%%%

\begin{remark} \label{Compare}
\em{
Let \( I \) be a monomial ideal in the polynomial ring 
\( R = K[x_1, \ldots, x_n] \). Suppose that \( n_0 \) denotes the 
index of stability of \( I \), and \( \ell_0 \) denotes the strong 
persistence index of \( I \). It is natural to ask whether these 
indices can be compared. In what follows, we present two examples 
showing that, in the first example, \( n_0 < \ell_0 \), while in the 
second example, \( n_0 > \ell_0 \). In other words, these two indices 
are not comparable in general.

\medskip
\textbf{First Example:} 
Consider the monomial ideal $I = (x_1^5 x_3,\; x_2^5,\; x_1 x_2^4,\; x_1^4 x_2)$ in the polynomial ring 
$R=K[x_1, x_2, x_3]$. We show that  the  index of stability of  $I$ is $2$, that is, $n_0=2$. It is routine to check that 
 $$I=(x_1, x_2^5) \cap (x_1^4, x_2^4) \cap (x_1^5, x_2) \cap (x_2, x_3).$$ 
This implies that $\mathrm{Ass}_R(R/I)=\mathrm{Min}(I)=\{(x_1, x_2),     (x_2, x_3)\}.$
In what follows, we establish  $\mathrm{Ass}_R(R/I^t)=\mathrm{Min}(I) \cup
\{(x_1, x_2, x_3)\}$ for all $t\geq 2$. Fix $t\geq 2$.  We claim $(I^t:_R f)=(x_1, x_2, x_3)$, where $f:=x_1^{4t-1} x_2^{t+2}$. To do this, we   consider the following statements:
\begin{itemize}
\item[(i)] By   $fx_1=x_1^{4t}x_2^{t+2}=(x_1^4x_2)^tx_2^2$ and $x_1^4x_2\in I$, this implies that  $fx_1\in I^t,$ and hence  $x_1 \in (I^t:_Rf)$;
\item[(ii)] As  $fx_2=x_1^{4t-1}x_2^{t+3}=(x_1^4x_2)^{t-1} (x_1^3x_2^4)$ and $x_1^3x_2^4\in I$, this gives that  $fx_2 \in I^t,$ and so   $x_2 \in (I^t:_Rf)$;
\item[(iii)] Since  $fx_3=x_1^{4t-1}x_2^{t+2}x_3 =(x_1^4x_2)^{t-2}(x_1^7x_2^4x_3)$ and $x_1^7x_2^4x_3\in I^2$, this yields that  $fx_3 \in I^t,$ 
and hence  $x_3 \in (I^t:_Rf)$.
\end{itemize} 
    Therefore,  $(x_1, x_2, x_3) \subseteq (I^t:_Rf).$ To finish our discussion,  we need to show that   $f\notin I^t$. On the contrary, assume that  $f \in I^t$. 
    This implies that there exist monomials $g_1, \ldots, g_t \in \mathcal{G}(I)$ such that   $x_3\nmid g_i$ for each $i=1, \ldots, t$, and  
    $g_1 \cdots g_t \mid   x_1^{4t-1} x_2^{t+2}$.   In particular, we must have the following equality
\begin{equation}
g_1 \cdots  g_t=(x_2^5)^{{\theta}_1} (x_1x_2^4)^{{\theta}_2} (x_1^4x_2)^{{\theta}_3}, 
\label{14}
\end{equation}
for some nonnegative integers $ {\theta}_1, \theta_2,  {\theta}_3$ with ${\theta}_1 +  \theta_2 +  {\theta}_3=t.$
 From  (\ref{14}), we get  
\begin{equation}
5\theta_1 + 4 \theta_2 + \theta_3 \leq t+2, \label{15}
\end{equation}
and 
\begin{equation}
\theta_2 + 4\theta_3  \leq  4t-1. \label{16}
\end{equation}
It follows from  $\sum_{i=1}^3 \theta_i =t$ and  (\ref{15}) that $t+4\theta_1+ 3\theta_2  \leq t+2$. This yields that  $4\theta_1+ 3\theta_2 \leq 2$. 
Hence, we must have $\theta_1= \theta_2=0$, and hence  $\theta_3 =t$. From  (\ref{16}), we obtain  $4t  \leq 4t-1$. This leads to a   contradiction. 
Consequently, we have $f \notin I^t$, and so $(I^t:_Rf) =(x_1, x_2, x_3)$. We thus get   $(x_1, x_2, x_3) \in 
\mathrm{Ass}_R(R/I^t)$ for all $t\geq 2$. Accordingly, we obtain $\mathrm{Ass}_R(R/I^t)=\mathrm{Min}(I) \cup
\{(x_1, x_2, x_3)\}$ for all $t\geq 2$. This implies that the  index of stability of  $I$ is $2$, that is, $n_0=2$. 

On the other hand, using Macaulay2 \cite{GS}, we obtain  $$(I^2:I)\neq I, \quad  (I^3:I)\neq I^2, \quad  (I^4:I)=I^3, \quad (I^5:I)=I^4.$$
Hence, we must have $\ell_0\geq 3$. This gives rise to $\ell_0 > n_0$. 
 \begin{verbatim}
i1 : S = QQ[x_1,x_2,x_3]

o1 = S

o1 : PolynomialRing

i2 : I = ideal(x_1^5*x_3, x_2^5, x_1*x_2^4, x_1^4*x_2)

o2 = ideal(x_1^5*x_3, x_2^5, x_1*x_2^4, x_1^4*x_2)

o2 : Ideal of S

i3 : associatedPrimes I

o3 = {ideal(x_1,x_2), ideal(x_2,x_3)}

o3 : List

i4 : associatedPrimes I^2

o4 = {ideal(x_1,x_2), ideal(x_2,x_3), ideal(x_1,x_2,x_3)}

o4 : List

i5 : I^2 : I == I

o5 = false

i6 : I^3 : I == I^2

o6 = false

i7 : I^4 : I == I^3

o7 = true

i8 : I^5 : I == I^4

o8 = true
\end{verbatim}

\medskip
\textbf{Second Example:}
 Assume that $C_{2n+1}=(V(C_{2n+1}), E(C_{2n+1}))$ is  a cycle graph on  vertex set $V(C_{2n+1})=\{x_1, \ldots, x_{2n+1}\}$ and edge set 
 $$E(C_{2n+1})=\{\{x_i, x_{i+1}\}~: ~ i=1, \ldots 2n\}\cup\{\{x_{2n+1}, x_1\}\}.$$
 Assume that \( R = K[x_1, \ldots, x_{2n+1}] \) is a polynomial ring, 
\( \mathfrak{m} = (x_1, \ldots, x_{2n+1}) \) is the maximal ideal of \(R\), 
and \( J(C_{2n+1}) \) denotes the cover ideal associated with \( C_{2n+1} \).
   It follows from \cite[Proposition 3.6]{NKA} that   $$\mathrm{Ass}_R(R/(J(C_{2n+1}))^s)= \mathrm{Ass}_R(R/J(C_{2n+1}))\cup \{\mathfrak{m}\},$$  for all $s\geq 2$. This yields that  the  index of stability of  $J(C_{2n+1})$ is $2$, that is, $n_0=2$. On the other hand, one can deduce from
    \cite[Theorem 3.3]{NKA} that $J(C_{2n+1})$ satisfies the strong persistence property. This means that $\ell_0=1$. We thus get $n_0>\ell_0$. 
    
   Accordingly,  these two indices  are not comparable in general.
}
\end{remark}

%%%%%%%%%%%%%%%%%%%%%%%%%%%%%%%%%%%%%%%%%%%%%%%%%%%%%%%%%%
%%%%%%%%%%%%%%%%%%%%%%%%%%%%%%%%%%%%%%%%%%%%%%%%%%%%%%%%%%

We conclude this section with the following results on the strong persistence index of ideals in commutative Noetherian rings.

\begin{proposition} \label{Pro.1}
Let $R$ be a commutative Noetherian ring and let $I$ be an ideal of $R$ whose strong persistence index is $\ell_0$. Then
\[
(I^r : I^s) = I^{r-s} \quad \text{for all } r - s \geq \ell_0.
\]
\end{proposition}

\begin{proof}
Take  $m \in (I^r : I^s)$. Then $m I^s \subseteq I^r$. Take any $h \in I^{s-1}$ and $g \in I$. Since $hg \in I^s$, it follows that
\[
(mh)g = m(hg) \in I^r \quad \text{for all } g \in I.
\]
Thus $mh \in (I^r : I)$ for all $h \in I^{s-1}$. Since $r \geq \ell_0 + 1$, we have
$(I^r : I) = I^{r-1},$
and hence $mh \in I^{r-1}$ for all $h \in I^{s-1}$. This implies that
  $m \in (I^{r-1} : I^{s-1}).$
Repeating this argument inductively $s-1$ times, we obtain
\[
m \in (I^{r-(s-1)} : I^{s-(s-1)}) = (I^{r-s+1} : I).
\]
Since $r - s + 1 \geq \ell_0 + 1$, it follows that $(I^{r-s+1} : I) = I^{r-s}.$ 
Therefore, $m \in I^{r-s}$, and the proof is complete.
\end{proof}

%++++++++++++++++++++++++++++++++++++++++++++++++++++++++++++++++++++++++++
%++++++++++++++++++++++++++++++++++++++++++++++++++++++++++++++++++++++++++
\begin{proposition} \label{Pro.2}
Let $I$ be an ideal of a commutative Noetherian ring $R$. Then
\[
\ell_0(I^k) < \frac{\ell_0(I)}{k} + 1
\]
for all integers $k \geq 1$.
\end{proposition}

\begin{proof}
Fix $k \geq 1$ and set $\ell_0' := \ell_0(I^k)$. By definition, we can deduce that 
$$((I^k)^{\ell_0'+1} : I^k)= (I^k)^{\ell_0'} \quad \text{and} \quad ((I^k)^{\ell_0'} : I^k) \neq (I^k)^{\ell_0'-1}.$$
Equivalently,
\[
(I^{k(\ell_0'+1)} : I^k)= I^{k\ell_0'}
\quad \text{and} \quad
(I^{k\ell_0'} : I^k) \neq I^{k(\ell_0'-1)}.
\]
Suppose, on the contrary,  that  $\ell_0(I^k) \geq \frac{\ell_0(I)}{k} + 1$. Then, this implies that 
 $k\ell_0' - k \geq \ell_0(I)$. Hence, from  Proposition \ref{Pro.1}, we get 
  $(I^{k\ell_0'} : I^k)= I^{k\ell_0' - k},$ which contradicts the minimality of $\ell_0'$. Hence, we must have 
  $k\ell_0' - k < \ell_0(I),$ and therefore $\ell_0' < \frac{\ell_0(I)}{k} + 1,$ as claimed. 
\end{proof}

%++++++++++++++++++++++++++++++++++++++++++++++++++++++++++++++++++++++++++
%++++++++++++++++++++++++++++++++++++++++++++++++++++++++++++++++++++++++++

\begin{proposition}
Let $I$ be an ideal of a commutative Noetherian ring $R$. Then
\[
\ell_0(I^k) \leq \ell_0(I)
\]
for all integers $k \geq 1$.
\end{proposition}

\begin{proof}
 From Proposition \ref{Pro.2}, we have the following inequalities
\[
\ell_0(I^k) < \frac{\ell_0(I)}{k} + 1 \leq \ell_0(I) + 1.
\]
Since $\ell_0(I^k)$ is an integer, it follows that $\ell_0(I^k) \leq \ell_0(I)$, as desired. 
\end{proof}

%%%%%%%%%%%%%%%%%%%%%%%%%%%%%%%%%%%%%%%%%%%%%%%%%%%%%%%%%%%%
%%%%%%%%%%%%%%%%%%%%%%%%%%%%%%%%%%%%%%%%%%%%%%%%%%%%%%%%%%%%

\section{On the  phenomenon  of fluctuations  in colon powers of monomial ideals} \label{Main Section 2}
 
This section is devoted to exploring the second topic of our study.
 Suppose that $I$ is a monomial ideal in a polynomial ring $R = K[x_1, \ldots, x_n]$. 
Another natural line of study in this context  concerns the behavior of colon ideals of powers of $I$ as follows:  

\begin{definition} \label{Def. Fluctuation}
Let $I \subset R = K[x_1, \ldots, x_n]$ be a monomial ideal. We say that $I$ has the  phenomenon  of \textit{fluctuation in colon powers}
 if there exist  positive integers $a<b<c$ such that at least one of the following cases occurs:
\begin{itemize}
\item[(i)]  $(I^{a}:I) = I^{a-1}, \quad (I^{b}:I) \neq I^{b-1}, \quad \text{but} \quad (I^{c}:I) = I^{c-1}.$
\item[(ii)] $(I^{a}:I) \neq I^{a-1}, \quad (I^{b}:I) = I^{b-1}, \quad \text{but} \quad (I^{c}:I) \neq I^{c-1}.$
\end{itemize}
\end{definition}

%+++++++++++++++++++++++++++++++++++++++++++++++++++++++++++++++++++++++++++++++++++

The following example demonstrates  that the behavior described in Definition   \ref{Def. Fluctuation} may occur for certain monomial ideals.

\begin{example} \label{Exam.Fluctuation}
\em{
\begin{itemize}
\item[(i)] Let  the  monomial ideal $I=(x^6, y^6, x y^5, x^5 y, x^4 y^4 )$ in $R=K[x,y]$.  
Using  \textit{Macaulay2} \cite{GS}, we deduce that 
$$(I^2 : I)= I, \quad  (I^3 : I)\neq I^2, \quad  (I^4 : I)\neq I^3, \quad  (I^5 : I)= I^4.$$
This  means that $I$ has the  phenomenon  of  fluctuation in colon powers.
\begin{verbatim}
i1 : S = QQ[x,y]

o1 = S

o1 : PolynomialRing

i2 : I = monomialIdeal(x^6,y^6,x*y^5,x^5*y,x^4*y^4)

o2 = monomialIdeal(x6,x5y,x4y4,xy5,y6)

o2 : MonomialIdeal of S

i3 : I^2 : I == I

o3 = true

i4 : I^3 : I == I^2

o4 = false

i5 : I^4 : I == I^3

o5 = false

i6 : I^5 : I == I^4

o6 = true
\end{verbatim}

\item[(ii)] Consider  the monomial ideal $J=(x^7, y^7, x^2 y^5, x^5 y^2) \subset R=K[x,y]$. By means of  \textit{Macaulay2} \cite{GS},  we obtain that 
$$(J^2 : J) \neq J, \quad (J^3 : J)= J^2, \quad (J^4 : J) \neq J^3, \quad  (J^5 : J) = J^4.$$
 We can now immediately deduce that $J$  exhibits the phenomenon of fluctuation in colon powers.
\begin{verbatim}
i1 : S = QQ[x,y]

o1 = S

o1 : PolynomialRing

i2 : J = monomialIdeal(x^7, y^7, x^2*y^5, x^5*y^2)

o2 = monomialIdeal(x7,x5y2,x2y5,y7)

o2 : MonomialIdeal of S

i3 : J^2 : J == J

o3 = false

i4 : J^3 : J == J^2

o4 = true

i5 : J^4 : J == J^3

o5 = false

i6 : J^5 : J == J^4

o6 = true
\end{verbatim}
 
\item[(iii)]  Assume  the following monomial ideal in $R=K[x,y,z]$, 
\begin{align*}
L=(&x^{8}y^{5},\,x^{7}y^{9},\,x^{6}y^{10},\,x^{4}y^{11},\,x^{13}y^{3}z^{2},\,x^{8}y^{4}z^{2},\,x^{6}y^{9}z^{2},\\
& x^{13}y^{2}z^{3},\,x^{4}y^{9}z^{3},\,x^{4}y^{4}z^{5},\,y^{5}z^{7},\,x^{13}z^{8}).
\end{align*}
 Using  \textit{Macaulay2} \cite{GS}, we deduce that 
$$(L^2 : L) = L,\quad (L^3 : L) \neq L^2,  \quad  (L^4 : L) = L^3, \quad  (L^5: L) = L^4.$$
Thus, $L$ has the  phenomenon  of  fluctuation in colon powers.
\begin{verbatim}
i1 : S = QQ[x,y,z]

o1 = S

o1 : PolynomialRing

i2 : L = monomialIdeal(
        x^8*y^5, x^7*y^9, x^6*y^10, x^4*y^11,
        x^13*y^3*z^2, x^8*y^4*z^2, x^6*y^9*z^2,
        x^13*y^2*z^3, x^4*y^9*z^3, x^4*y^4*z^5,
        y^5*z^7, x^13*z^8
     )

o2 = monomialIdeal(x8y5,x7y9,x6y10,x4y11,x13y3z2,x8y4z2,x6y9z2,
                   x13y2z3,x4y9z3,x4y4z5,y5z7,x13z8)

o2 : MonomialIdeal of S

i3 : L^2 : L == L

o3 = true

i4 : L^3 : L == L^2

o4 = false

i5 : L^4 : L == L^3

o5 = true

i6 : L^5 : L == L^4

o6 = true
\end{verbatim}
\end{itemize}
}
\end{example}

%%%%%%%%%%%%%%%%%%%%%%%%%%%%%%%%%%%%%%%%%%%%%%%%%%%%
%%%%%%%%%%%%%%%%%%%%%%%%%%%%%%%%%%%%%%%%%%%%%%%%%%%%

It is natural to ask whether there exist infinitely many monomial ideals in $R=K[x_1, \ldots, x_n]$, where $n\geq 2$,  that exhibit the phenomenon of fluctuation in colon powers.
 We provide a positive answer to this question in Theorem \ref{Th.Fluctuation}. To achieve this,  we first establish the following auxiliary lemma.

\begin{lemma}\label{Lem.Fluctuation} 
Let  $I \subset R=K[x_1, \ldots, x_n]$ be a monomial ideal.  Then the following statements hold.
\begin{itemize}
\item[(i)] $I$ has the  phenomenon  of  fluctuation in colon powers   if and only if $I^*$ has the  phenomenon  of  fluctuation in colon powers,
  where $I^*$ denotes the  expansion of $I$.  
\item[(ii)] $I$ has the  phenomenon  of  fluctuation in colon powers  if and only if $I_W$ has the  phenomenon  of  fluctuation in colon powers, 
  where $I_W$ denotes the  weighted ideal.    
\end{itemize}
\end{lemma}

\begin{proof}
(i) Recall that  for any two monomial ideals $A$ and $B$, it follows from Lemma \ref{Lem.Bayati.Expansion}(i)  that $A=B$ if and only if $A^*=B^*$. 
Now, the claim can be shown by parts (iii) and (v) of  Lemma \ref{Lem.Bayati.Expansion}.

(ii)   Remember that   for any two monomial ideals $A$ and  $B$, we always have $A=B$ if and only if $A_W=B_W$. Hence,  the assertion can be verified 
 by parts (ii) and (iv) of  Lemma \ref{LEM. Weighted}.
\end{proof}

%%%%%%%%%%%%%%%%%%%%%%%%%%%%%%%%%%%%%%%%%%%%%%%%%%%%%%%%%%%%
%%%%%%%%%%%%%%%%%%%%%%%%%%%%%%%%%%%%%%%%%%%%%%%%%%%%%%%%%%%%

We are now ready to state and prove the main result of this section in the following theorem. 

\begin{theorem} \label{Th.Fluctuation}
Let $R=K[x_1, \ldots, x_n]$ with $n\geq 2$ be  a polynomial ring in $n$  variables  with coefficients in a  field $K$. 
Then there exist infinitely many  monomial ideals  in $R$  that  have  the  phenomenon  of  fluctuation in colon powers. 
\end{theorem}

\begin{proof}
Let $J$ be a monomial ideal in $R =K[x,y]$ exhibiting the phenomenon of fluctuation in colon powers
 (see Example \ref{Exam.Fluctuation}(ii), which guarantees the existence of such a monomial  ideal). 
 Here, we consider  $$\mathfrak{p}_1:=(x_{i_1}, \ldots, x_{i_a}) \quad \text{and} \quad  \mathfrak{p}_2:=(x_{i_{a+1}}, \ldots, x_{i_b}),$$
  with  $\mathrm{supp}(\mathfrak{p}_1) \cap  \mathrm{supp}(\mathfrak{p}_2)=\emptyset$   and
  $\mathrm{supp}(\mathfrak{p}_1) \cup \mathrm{supp}(\mathfrak{p}_2) =\{x_1, \ldots, x_n\}$. 
 It follows at once from  Lemma  \ref{Lem.Fluctuation}(i) that  $J^*$ in $R=K[x_1, \ldots, x_n]$
 has the phenomenon  of  fluctuation in colon powers. Now, set   $I:=(J^*)_W$ with $W(x_i)=\alpha_i$ such that $\alpha_i \geq 1$ for all $i=1, \ldots, n$. 
 According to Lemma  \ref{Lem.Fluctuation}(ii), we can conclude that   $I$  has the phenomenon  of  fluctuation in colon powers as well. 
 This implies that    there exist infinitely many  monomial ideals  in     $R=K[x_1, \ldots, x_n]$    that  have  the  phenomenon  of
   fluctuation in colon powers, as claimed.  
\end{proof}

%++++++++++++++++++++++++++++++++++++++++++++++++++++++++++++++++++++++++++
%++++++++++++++++++++++++++++++++++++++++++++++++++++++++++++++++++++++++++

\section{Conclusion and Further Directions}

In this paper, we studied two related aspects of the behavior of powers of ideals in commutative Noetherian rings. 

First, we introduced and investigated the notion of the strong persistence index of an ideal. In particular, we showed that there exist infinitely many monomial ideals in $R=K[x_1,\ldots,x_n]$, with $n \geq 3$, sharing the same strong persistence index $\ell_0 \geq 2$. This demonstrates that the strong persistence index can exhibit rich and nontrivial behavior even within well-structured classes of monomial ideals.

Second, we examined the phenomenon of fluctuation in colon powers. We proved that there exist infinitely many monomial ideals in $K[x_1,\ldots,x_n]$, with $n \geq 2$, whose colon powers do not stabilize monotonically, but instead exhibit alternating behavior. This shows that colon stability properties can fail in a controlled and recurring manner.

These results suggest several directions for further research. One natural problem is to study the strong persistence index beyond the monomial setting, for example in binomial or more general classes of ideals. Another interesting direction is to classify ideals exhibiting fluctuation in colon powers and to understand how this behavior interacts with other homological or asymptotic invariants, such as depth functions, associated primes, or Rees algebra properties.

We hope that the methods developed here will contribute to a deeper understanding of asymptotic colon behavior in commutative algebra.

%++++++++++++++++++++++++++++++++++++++++++++++++++++++++++++++++++++++++++
%++++++++++++++++++++++++++++++++++++++++++++++++++++++++++++++++++++++++++

\bigskip
\noindent{\bf Acknowledgments} \par 
The authors are deeply grateful to the four anonymous referees for their careful reading of the manuscript and for their valuable suggestions and comments, which led to several improvements in the quality of this paper. In particular, the second author, Jonathan Toledo, is partially supported by SECIHTI-CBF-2025-I-1583.

%++++++++++++++++++++++++++++++++++++++++++++++++++++++++++++++++++++++++++
%++++++++++++++++++++++++++++++++++++++++++++++++++++++++++++++++++++++++++

 \bigskip
\noindent\textbf{Data availability statement}\par
This manuscript has no associated data.  
%++++++++++++++++++++++++++++++++++++++++++++++++++++++++++++++++++++++++++
%++++++++++++++++++++++++++++++++++++++++++++++++++++++++++++++++++++++++++

\bigskip
\noindent\textbf{Disclosure statement}\par 
The authors declare no financial or non-financial competing interests.  
%++++++++++++++++++++++++++++++++++++++++++++++++++++++++++++++++++++++++++
%++++++++++++++++++++++++++++++++++++++++++++++++++++++++++++++++++++++++++

%\bigskip
%\textbf{Funding}\par 
%The first author was partially supported by a grant from ABC (No.\ 1234567890).  
%++++++++++++++++++++++++++++++++++++++++++++++++++++++++++++++++++++++++++
%++++++++++++++++++++++++++++++++++++++++++++++++++++++++++++++++++++++++++

\bigskip
\noindent\textbf{ORCID}
\begin{itemize}
\item  {0000-0003-1073-1934:}~~~~{Mehrdad Nasernejad}
\item {0000-0003-3274-1367:}~~~~{Jonathan Toledo}
\end{itemize}

%%%%%%%%%%%%%%%%%%%%%%%%%%%%%%%%%%%%%%%%%%%%%%%%%%%%%%%%%%%
%%%%%%%%%%%%%%%%%%%%%%%%%%%%%%%%%%%%%%%%%%%%%%%%%%%%%%%%%%%
%%%%%%%%%%%%%%%%%%%%%%%%%%%%%%%%%%%%%%%%%%%%%%%%%%%%%%%%%%%
%%%%%%%%%%%%%%%%%%%%%%%%%%%%%%%%%%%%%%%%%%%%%%%%%%%%%%%%%%%

% ------------------------------------------------------------------------
\end{document}